\documentclass[12pt]{amsart}

\usepackage{amssymb}
\usepackage{amsmath}
\usepackage{amsfonts}
\usepackage[all]{xy} 
\usepackage{mathrsfs,enumitem}
\usepackage[utf8]{inputenc}
\usepackage{color}
\usepackage{ulem} 

\newcommand{\disk}{\ensuremath{\mathbb{D}} } 

\newcommand{\sphere}{\overline{\mathbb{C}}}

\theoremstyle{plain}
        \newtheorem{theorem}{Theorem}[section]
        \newtheorem{Problem}{Problem}
        \newtheorem{lemma}[theorem]{Lemma}
        
        \newtheorem{corollary}[theorem]{Corollary}

\theoremstyle{definition}
        \newtheorem{definition}[theorem]{Definition}

\theoremstyle{remark}
    \newtheorem{remark}[theorem]{Remark}

\numberwithin{equation}{section} 
\numberwithin{figure}{section} 

\usepackage{fullpage}

\author{Eric Schippers}
\address{Eric Schippers \\ Department of Mathematics \\
University of Manitoba\\
Winnipeg, Manitoba \\  R3T 2N2 \\ Canada}
\email{eric\_schippers@umanitoba.ca}

\author{Wolfgang Staubach}
\address{Wolfgang Staubach\\ Department of Mathematics\\
Uppsala University\\
Box 480\\ 751 06 Uppsala\\ Sweden}
\email{wulf@math.uu.se}



\title[Harmonic reflection and Riemann-Hilbert problem]{Harmonic reflection in quasicircles and well-posedness of a Riemann-Hilbert problem on quasidisks}

\dedicatory{Dedicated to the memory of Ida Bulat}

\begin{document}

\begin{abstract}  A complex harmonic function of finite Dirichlet energy on a Jordan domain has
 boundary values in a certain conformally invariant sense, by a construction of H. Osborn.
 We call the set of such boundary values the Douglas-Osborn space.
 One may then attempt to solve the Dirichlet problem on the complement for these boundary values.  This defines a
 reflection of harmonic functions.  We show that quasicircles are precisely those Jordan curves
 for which this reflection is defined and bounded.

 We then use a limiting Cauchy integral along level curves of Green's function to show that
 the Plemelj-Sokhotski jump formula holds on quasicircles with boundary data in the Douglas-Osborn
 space. This enables us to prove the well-posedness of a Riemann-Hilbert problem with boundary data in the Douglas-Osborn
 space on quasicircles.
\end{abstract}
\thanks{Both authors are grateful for the financial support from the Wenner-Gren Foundations. Eric Schippers is also partially supported by the National Sciences and Engineering Research Council of Canada. }
\maketitle

\begin{section}{Introduction}

Consider the following Riemann-Hilbert problem:
\begin{Problem}\label{Riemann-Hilbert}
Given a domain in the plane and a function defined on its boundary, find holomorphic functions in the inside and the outside the domain so that the difference of their boundary values is equal to the given function.
\end{Problem}

The main issue in this problem is the regularity of the boundary value function (i.e. the boundary data) on one hand, and the regularity of the boundary curve on the other. If the boundary of the domain is a smooth simple closed curve and the boundary function is  Lipschitz on the boundary curve, then Problem 1 is the famous Sokhotski-Plemelj jump problem, which was solved independently by Y. Sokhotski \cite{Sokhotski} and J. Plemelj \cite{Plemelj}. In general, if the boundary of the domain is a rough, say fractal-type curve, then given certain functions on the domain, it is not clear how to extend those functions to the boundaries. However, in \cite{Semmes}, S. Semmes studied Problem 1 when the boundary curve is a so-called chord-arc curve, and obtained a solution to the problem when the boundary data is in $L^2$. Another difficulty that frequently occurs in this context is that Sobolev/Besov spaces are not well-defined on all fractal-type curves, which makes the boundary value problems with data in those spaces problematic. Furthermore, it is difficult to establish the existence of boundary values of functions that are defined on either of the sides of such curves, or even to find the appropriate sense in which they are defined. Nevertheless, in problems connected to quasiconformal Teichm\"uller theory as well as those arising in connection to the mathematical foundations of conformal field theory, one inevitably confronts Problem 1 where the domain is a quasidisk. Indeed our particular choice of the Riemann-Hilbert problem here is strongly motivated by applications in the aforementioned fields \cite{Radnell-Schippers-Staubach}.\\

A quasidisk is the image of the unit disk under a quasiconformal map of the plane. The boundaries of quasidisks, called quasicircles, can be very rough curves with Hausdorff dimensions arbitrarily close to $2$; examples include fractal-type curves such as snowflakes. The authors (together with D. Radnell) had previously studied Problem 1 in the case when the boundary curve is assumed to be a Weil-Petersson class quasicircle \cite{RSS_WPjump}, and Schippers and Staubach extended the study of Problem 1 to $d$-regular quasicircles (which are a generalization of Ahlfors-regular quasicircles) where the boundary function was assumed to belong to a certain Besov space \cite{SchippersStaubachBesov}.\\

Here we consider Problem \ref{Riemann-Hilbert} when the domain is a general quasidisk and remove the assumption of $d$-regularity which was imposed in \cite{SchippersStaubachBesov}. Moreover, as opposed to \cite{SchippersStaubachBesov}, we obtain the solution for an optimal class of boundary data and show that our solution yields yet another characterization of quasidisks. More precisely, to bypass the regularity issues of the boundary, we use a particular reflection of harmonic functions, and a limiting Cauchy integral. The reflection of harmonic functions exists and is bounded if and only if the domain is a quasidisk.  We then use this to  demonstrate that Problem \ref{Riemann-Hilbert} has a unique solution for a space of functions (which we call the Douglas-Osborn space) defined on the boundary of a quasidisk, and that the solution depends continuously on the boundary data. Hence the Riemann-Hilbert problem above is well-posed.\\

 In this connection we would like to mention that the solution of this boundary value problem could be compared with Poincar\'e's lemma. That is, the existence of solutions to $d f=\alpha$ for a closed differential form $\alpha$ (exactness) is intimately connected to the topology of the underlying domain (or manifold). Here, the existence of a solution of the Riemann-Hilbert problem is intimately connected, rather, to the regularity (e.g. the Hausdorff dimension) of its boundary.\\

   Here is an explanation of our approach.  Recall the classical Plemelj-Sokhotski jump formula \cite{Plemelj}, which asserts that for a complex function $h$ (say, continuous)
 on a smooth closed Jordan curve $\Gamma$ bounding two regions $\Omega^\pm$ in the plane (where
 we assume $\Omega^-$ is unbounded), if
 \begin{equation} \label{eq:intro}
  h_\pm(z) = \pm \frac{1}{2\pi i} \int_{\Gamma} \frac{h(\zeta)}{\zeta - z} d\zeta   \ \ \ z \in \Omega^\pm
 \end{equation}
 then
 \[  h(z) = \lim_{\Omega^+ \ni w \rightarrow z} h_+(w) + \lim_{\Omega^- \ni w \rightarrow z} h_-(w).  \]
 For smooth boundaries and regular boundary values,
 this provides a solution to Problem 1, as mentioned above; the term ``jump'' refers to the discontinuity in the boundary values of the Cauchy integral.  Henceforth, we will use the
 equivalent formulation of the jump problem which requires an expression for $h$ as a sum
 rather than a difference.
 To solve Problem \ref{Riemann-Hilbert} for quasicircles, we prove a jump formula for boundary data that are the limiting values of complex harmonic functions of bounded Dirichlet energy. The choices
 are naturally dictated and in some sense optimal.  We also give a characterization of quasidisks
 in terms of the existence of a bounded reflection on harmonic functions.\\

 The boundary values are defined using the following conformally
 invariant construction of H. Osborn \cite{Osborn}.  Let $\Gamma$ be any
 Jordan curve, with Green's function $g_p$ with singularity at $p$.  Let $C_{p,\epsilon}$
 denote the level curves $g_p(z)=-\epsilon$, and for $q \in \Gamma$ let $\gamma_{p,q}$ denote
 the unique orthogonal curve beginning at $p$ and terminating at $q$.
 Osborn showed that for any Jordan
 curve, a harmonic function of finite Dirichlet energy in a connected component of the complement
 has boundary values a.e. on $\Gamma$ with respect to harmonic measure, where these boundary values are approached along $\gamma_{p,q}$.

 In this paper, we show that if $\Gamma$ is a quasicircle bounding two domains $\Omega^\pm$ in the sphere,
 the space of functions obtained as boundary values of complex harmonic functions
 of finite Dirichlet energy on $\Omega^+$ or $\Omega^-$ are the same.   We call this
 space the Douglas-Osborn space.   Thus we may
 define a ``harmonic reflection'' in quasicircles as follows.  Given $h_{\Omega^+}$ in the Dirichlet
 space of $\Omega^+$, restrict to $\Gamma$ in the sense of Osborn, and then let $h_{\Omega^-}$ be
 the unique function in the Dirichlet space of $\Omega^-$ with boundary values $h$.
 This is well-defined for quasicircles by the above, and furthermore is bounded.  The same holds for
 the inverse reflection.\\

 We also define a boundary Cauchy integral  for $h$ in the Douglas-Osborn space of a quasicircle
 $\Gamma$,
 by performing the Cauchy integral (\ref{eq:intro}) of the harmonic extension of $h$ to $\Omega^+$ along the
 level curves $C_{p,\epsilon}$
 and letting $\epsilon$ go to zero.  We show that one obtains the same result if one extends
 instead to $\Omega^-$ and uses level curves in $\Omega^-$.  Finally, we show that the outcome
 of this limiting Cauchy integral satisfies the Plemelj-Sokhotski jump formula.\\

 We conclude the introduction with a remark on the results of Osborn. Osborn's work was concerned with extending the ideas
of J. Douglas used in his resolution of the Plateau problem and making
them rigorous. In particular,
 Douglas expressed the Dirichlet norm of a harmonic function in terms of an integral over
 the boundary of a domain, and in some sense his solution was closely related to the Dirichlet principle.
 For a historical discussion see J. Gray and M. Micallef \cite{GrayMicallef}. Osborn's beautiful
 paper extended this to finitely-connected domains bounded by Jordan curves, among other striking results.
 It is with this historical development in mind that we named the set of boundary values
 after Douglas and Osborn.\\

\end{section}
\begin{section}{Reflection of harmonic functions}  \label{se:reflection}
\begin{subsection}{Notation and terminology}
  In this paper, $\mathbb{C}$ is the complex plane,
 $\sphere$ is the Riemann sphere, $\disk^+ = \{z \in \mathbb{C} \,:\, |z| <1 \}$,
 and $\disk^- = \{z  \in \mathbb{C} \,:\, |z|>1 \} \cup \{ \infty \}$.
 We will denote the boundary of a domain $\Omega$ by $\partial \Omega$, and
 $\partial \disk^\pm$ by $\mathbb{S}^1$.   Closure will be denoted by $\text{cl}$. Furthermore, generic constants appearing in all the estimates will be denoted by $C$ even though they may be different in the same line or from line to line.

 Throughout the paper, we will deal with Jordan curves $\Gamma$ in $\sphere$.
  We always assume that $\Gamma$ is given an orientation, and denote the components
  of the complement of $\Gamma$ in $\sphere$ by $\Omega^+$ and $\Omega^-$.
  We choose $\Omega^+$ to be the component such that $\Gamma$ is positively oriented with respect to $\Omega^+$.
  In the case that $\Gamma$ is bounded, we always assume that $\Omega^+$ is the
  bounded component of the complement, so that the orientation of $\Gamma$ is uniquely
  determined in this case.
  In Section \ref{se:reflection}, the orientation of the curve is only necessary to fix a meaning of the
  notation $\Omega^\pm$.  In Section \ref{se:jump} the orientation affects the signs of the integrals.
  Finally, the phrase ``$\Gamma$ bounds $\Omega$'' does
  not imply that $\Omega$ is a bounded domain.

We shall also consider harmonic and holomorphic functions on domains containing
  $\infty$.  To say that $h$ is harmonic or holomorphic on a domain $\Omega$ containing
  $\infty$ is to say that $h$ is harmonic or holomorphic on $\Omega \backslash \{\infty \}$
  and $h(1/z)$ is harmonic or holomorphic on a neighbourhood of $0$.
  In this paper, a conformal map between open connected subsets $D$ and $\Omega$ of $\sphere$ is taken to mean a map which is a meromorphic bijection between $D$ and $\Omega$.  (In particular, it can have at worst a simple pole, possibly at $\infty$).
\end{subsection}
\begin{subsection}{Osborn space of boundary values}
Throughout
 the paper, we choose norms which are conformally invariant, so that the behaviour
 at $\infty$ does not require special treatment.

Let $dA$ denote the Lebesgue measure, and $\partial=\frac{1}{2}( \partial_x -i\partial_y )$, $\overline {\partial}=\frac{1}{2} (\partial_x +i\partial_y )$ denote the Wirtinger operators. Define the complex Dirichlet space by
 \[  \mathcal{D}_{\mathrm{harm}}(\Omega) = \left\{ h\,\,\,  \text{is  harmonic } \,:\, \iint_{\Omega} \left(|\partial h|^2 + |\overline{\partial}h|^2
 \right)dA < \infty
   \right\}  \]
 with semi-norm
 \begin{equation} \label{eq:Dharm_seminorm}
   \| h \|^2_{\mathcal{D}_{\mathrm{harm}}(\Omega)} =  \iint_{\Omega} \left( |\partial h|^2 + |\overline{\partial}h|^2 \right)dA.
 \end{equation}
 This is of course not a norm since if two functions $h_1$ and $h_2$ differ by a constant, then
 $\| h_1\|_{\mathcal{D}_{\mathrm{harm}}(\Omega)}=\| h_2\|_{\mathcal{D}_{\mathrm{harm}}(\Omega)}$.
 Note that this space is conformally invariant, so that if $g:D \rightarrow \Omega$ is any biholomorphism
 then
 \[  \| h \circ g \|_{\mathcal{D}_{\mathrm{harm}}(D)} = \| h \|_{\mathcal{D}_{\mathrm{harm}}(\Omega)}.  \]

 For $p \in \Omega$, we will also consider the pointed Dirichlet space $\mathcal{D}_{\mathrm{harm},p}(\Omega)$,
 which is the same set of functions but with the actual norm
  \[  \| h \|^2_{\mathcal{D}_{\mathrm{harm},p}(\Omega)} =  |h(p)|^2 + \iint_{\Omega} \left(|\partial h|^2 + |\overline{\partial} h|^2
    \right) dA.  \]
 This also satisfies a conformal invariance:
 if $g:D \rightarrow \Omega$ is any biholomorphism
 then
 \[  \| h \circ g \|_{\mathcal{D}_{\mathrm{harm},p}(D)} = \| h \|_{\mathcal{D}_{\mathrm{harm},g(p)}(\Omega)}.  \]

 Now let $\Omega$ be a Jordan domain as above with boundary $\Gamma$.  We define a
 set of boundary values of $\mathcal{D}_{\mathrm{harm}}(\Omega)$, following Osborn \cite{Osborn}, as follows.
 First some notation is required.
 Fix a point $p \in \Omega$ and let $g_p$ be Green's function on $\Omega$ with
 singularity at $p$.   For $\epsilon>0$ let $C_{p,\epsilon}$ be the level curve $g_p(z)= - \epsilon$.
 For $q \in \Gamma$, let $\gamma_{p,q}$ denote the unique ray which is orthogonal to the set of
 curves $C_{p,\epsilon}$, starting at $p$ and terminating at $q$.

 Another characterization of the curves $\gamma_{p,q}$ and $C_{p,\epsilon}$ will be useful.  Let
 $f:\disk^+ \rightarrow \Omega$ be a conformal map such that $f(0)=p$.  By the Osgood-Carath\'eodory
 theorem, $f$ extends to a homeomorphism of $\mathbb{S}^1$ onto $\Gamma$.    In that case,
 each curve $\gamma_{p,q}$ is the image of the ray $t \mapsto t f^{-1}(q)$ for $t \in [0,1]$.
 Also $C_{p,\epsilon}$ is the image of $|z|=r$ under $f$ for $r = e^{-\epsilon}$.
 This follows directly from the conformal invariance of Green's function. One
 can also characterize the curves $\gamma_{p,q}$ as the hyperbolic geodesic ray
 with end-points $p$ and $q$.

 In \cite{Osborn}, Osborn showed that for any $h \in \mathcal{D}_{\mathrm{harm}}(\Omega)$, the limit
 \[  \lim_{z \rightarrow q} h(z) \]
 exists almost everywhere.  We provide the proof in the simply connected case,
 as a consequence of classical results.\\

 The Sobolev space $H^{\frac{1}{2}}(\mathbb{S}^1)$ is the space of functions $h\in L^2 (\mathbb{S}^1)$ such that $\sum_{n=-\infty}^{\infty}
 (1+|n|^2)^{\frac{1}{2}}|\hat{h}(n)|^2 <\infty ,$  where $\hat{h}(n):=\frac{1}{2\pi}\int_{\mathbb{S}^1} h(t) e^{-int}\, dt.$ The norm of $h$ in $H^{\frac{1}{2}}(\mathbb{S}^1)$ is given by
 \[  \Vert h\Vert^2_{H^{\frac{1}{2}}(\mathbb{S}^1)}=\sum_{n=-\infty}^{\infty} (1+|n|^2)^{\frac{1}{2}}|\hat{h}(n)|^2 .  \]
 Here we consider a slight modification of this space, denoted by $\mathcal{H}(\mathbb{S}^1)$, which consists of $L^2$ functions on the
 circle for which
 \begin{equation} \label{eq:circle_boundary_seminorm}
  \|h\|^2_{\mathcal{H}(\mathbb{S}^1)} =  \sum_{n=-\infty}^\infty |n| |\hat{h}(n)|^2 <\infty.
 \end{equation}
 This is of course a semi-norm.  $\mathcal{H}(\mathbb{S}^1)$ and
 $H^{\frac{1}{2}}(\mathbb{S}^1)$ are the same set of functions \cite{SchippersStaubach_Proc}.
 We also consider the norm
 \begin{equation} \label{eq:circle_boundary_norm}
  \|h\|^2_{\mathcal{H}_0(\mathbb{S}^1)} = |\hat{h}(0)|^2 + \sum_{n=-\infty}^\infty |n| |\hat{h}(n)|^2 <\infty.
 \end{equation}
 One can easily show that the norms of  $H^{\frac{1}{2}}(\mathbb{S}^1)$ and $\mathcal{H}_0(\mathbb{S}^1)$ are equivalent, see e.g. \cite{SchippersStaubach_Proc}.

In the following discussion, and the rest of this paper, the term ``capacity'' will always refer to logarithmic capacity, and similarly for ``outer capacity''.
The space $\mathcal{H}(\mathbb{S}^1)$ consists precisely of the boundary values
 of functions in the Dirichlet space $\mathcal{D}_{\text{harm},0}(\disk^+)$, and the norms
 are equal \cite{NagSullivan}.  We briefly outline this correspondence here.
 The Fourier series of any element $h(e^{i\theta})= \sum_{n=-\infty}^\infty \hat{h}(n) e^{i n \theta} \in \mathcal{H}(\mathbb{S}^1)$
 converges everywhere on $\mathbb{S}^1$ except possibly on a set of outer capacity zero, see A. Zygmund, \cite[Theorem 11.3, Chapter XIII]{Zygmund_book}.    Therefore
  the series
  \begin{equation} \label{eq:H_form_Dirichlet}
   H(z)= \sum_{n=0}^\infty \hat{h}(n) z^n + \sum_{n=1}^\infty \hat{h}(-n) \bar{z}^n
  \end{equation}
  converges to a sum of a holomorphic and an anti-holomorphic function and is thus complex
  harmonic.  Furthermore, a standard computation shows that
  \[  \| h \|_{\mathcal{H}(\mathbb{S}^1)} = \| H \|_{\mathcal{D}_{\text{harm}}(\disk^+)}
   \ \ \text{and} \ \  \| h \|_{\mathcal{H}_0(\mathbb{S}^1)} = \| H \|_{\mathcal{D}_{\text{harm},0}(\disk^+)}.  \]
   By Abel's theorem and the convergence of $h$ in the sense above, $\lim_{t \rightarrow 1^-}{H(te^{i\theta})} = h(e^{i\theta})$ for
   all $e^{i\theta} \in \mathbb{S}^1$ except on a set of zero outer capacity; that is,
   the boundary values agree with $h$ in the sense of Osborn except on this set.  Conversely,
   an element $H(z) \in \mathcal{D}_{\text{harm}}(\disk^+)$ has the form
   (\ref{eq:H_form_Dirichlet}) with coefficients $\hat{h}(n)$ satisfying
   (\ref{eq:circle_boundary_seminorm}).  One thus obtains an element $h$ of $\mathcal{H}(\mathbb{S}^1)$ by replacing $z$ with $e^{i\theta}$, and the argument
   above can be repeated to show that $h$ is the limit of $H$ on the boundary
   in the same sense.

Summarizing, we have the following result.
 \begin{theorem} \label{th:boundary_values_disk}
  Every $H \in \mathcal{D}_{\mathrm{harm}}(\disk^+)$ has boundary values $($in the sense of Osborn$)$ except on a subset
   of $\mathbb{S}^1$ with outer capacity zero, and the restriction to the boundary belongs to $\mathcal{H}(\mathbb{S}^1)$.
  Conversely, every $h \in \mathcal{H}(\mathbb{S}^1)$ has a unique complex harmonic extension
  $H \in \mathcal{D}_{\mathrm{harm}}(\disk^+)$ whose boundary values coincide with $h$
  in the same sense. The restriction and extension are isometries with respect to $\| \cdot \|_{\mathcal{H}_0(\mathbb{S}^1)}$
  and $\| \cdot \|_{\mathcal{D}_{\mathrm{harm},0}(\disk^+)}$. The restriction and
  extension are also bounded with respect to the  semi-norms $\| \cdot \|_{\mathcal{H}(\mathbb{S}^1)}$
  and $\| \cdot \|_{\mathcal{D}_{\mathrm{harm}}(\disk^+)}$.
 \end{theorem}

 \begin{remark} \label{re:outer_cap_contained} Any set $E \subset \mathbb{S}^1$ of outer capacity zero is
  contained in a Borel set $I \subset \mathbb{S}^1$ of capacity zero.  To see this, for any $n$ let $V_n$ be an open subset of $\mathbb{C}$ containing $E$ such that the capacity of $V_n$ is
  less than $1/n$. If we let $U_n = V_n \cap \mathbb{S}^1$, then $U_n$ is an open set in $\mathbb{S}^1$ (and
  a Borel set in $\mathbb{C}$).  Since $U_n \subset V_n$ the capacity of $U_n$ is less than $1/n$ by
  \cite[5.1.2(a)]{Ransford_book}.  Now if we set $I = \cap_{n=1}^\infty U_n$ (which is obviously a Borel set), then
  since $I \subseteq U_n$ for all $n$ (again using \cite[5.1.2(a)]{Ransford_book}) we have that the capacity of
  $I$ is less than or equal to $1/n$ for all $n$.  Hence $I$ is the desired set of capacity zero which contains $E$ by
  construction.

  Note that it is possible to take $I$ to be a $G_\delta$ set.  To see this,
  in the above construction one could choose $U_n = V_n \cap \{ z\,:\, 1-1/n < |z|< 1+ 1/n \}$.
 \end{remark}

 The extension and restriction map is described explicitly as follows.
 Given $h \in \mathcal{D}_{\mathrm{harm}}(\disk^+)$, it has a decomposition $h = u + \overline{v}$
 for holomorphic functions
 \[  u(z) = \sum_{n=0}^\infty u_n z^n, \ \   \ \ \ v(z) = \sum_{n=1}^\infty v_n z^{n}.  \]
 In that case $h(e^{i\theta}) = \sum_{n=0}^\infty u_n e^{in \theta} + \sum_{n=1}^\infty \overline{v_n} e^{-i n \theta}$
 and the extension is obtained similarly.

 \begin{remark} \label{re:equal_on_Dplusminus}
 We also observe that Theorem \ref{th:boundary_values_disk} applies equally to $\mathcal{D}_{\mathrm{harm}}(\disk^-)$.  Note
  that the reflection
  \begin{align*}  \label{eq:reflection_in_circle}
   \mathfrak{R}(\mathbb{S}^1):\mathcal{D}_{\mathrm{harm}}(\disk^+) & \rightarrow \mathcal{D}_{\mathrm{harm}}(\disk^-)  \\
   H(z) & \mapsto H(1/\bar{z})
  \end{align*}
  is a semi-norm preserving isomorphism between $\mathcal{D}_{\mathrm{harm}}(\disk^+)$ and $\mathcal{D}_{\mathrm{harm}}(\disk^-)$.
  In fact it is an isometry between $\mathcal{D}_{\mathrm{harm},0}(\disk^+)$ and $\mathcal{D}_{\mathrm{harm}, \infty}(\disk^-)$.

  We will also denote the inverse of $\mathfrak{R}(\mathbb{S}^1)$ with the same notation.
 \end{remark}

 Thus given a Jordan domain $\Omega$ and fixed $p \in \Omega$ we obtain
 boundary values for $H \in \mathcal{D}_{\mathrm{harm}}(\Omega)$ as follows.
 \begin{theorem}[Osborn \cite{Osborn}] \label{Osborns sats} Let $\Gamma$ be a Jordan curve in $\sphere$
  and let $\Omega$ be either of the components of the complement.     Fix $p \in \Omega$.
  For any $H \in \mathcal{D}_{\mathrm{harm}}(\Omega)$,
  \[  h(q) = \lim_{z \rightarrow q} H(z), \ \ \ \ z \in \gamma_{p,q} \]
  exists for almost all $q$ with respect to the harmonic measure on
  $\Gamma$ with respect to $p$ and $\Omega$.
 \end{theorem}
 \begin{proof} Let $f:\disk^+ \rightarrow \Omega$ be a conformal map such that
  $f(0)=p$.  Since $f$ extends homeomorphically to a map $F:\text{cl} \, \disk^+
  \rightarrow \text{cl} \, \Omega$,
  we have that the limits
  \[  \lim_{\gamma_{p,q}\ni z \rightarrow q} H(z) = \lim_{f^{-1}(\gamma_{p,q}) \ni w \rightarrow f^{-1}(q)} H \circ F(w)       \]
  exist except on a set of outer capacity zero on $\mathbb{S}^1$ by Theorem \ref{th:boundary_values_disk},
  and thus also exist except on a set of capacity zero by Remark \ref{re:outer_cap_contained}. A set of capacity zero on $\mathbb{S}^1$ has zero harmonic measure with respect to $\disk^+$ (see T. Ransford \cite[Theorem 4.3.6]{Ransford_book}).  Since the harmonic measure of a set $E \subset \partial \Omega$
  with respect to $\Omega$ is the harmonic measure of $F^{-1}(E)$ with respect to
  $\mathbb{S}^1$ (see J. Garnett and D. Marshall
  \cite[I.3]{Garnett_Marshall}), this completes the proof.
 \end{proof}
 \begin{remark} \label{re:conformal_invariance_limits}  Note that the proof shows that the set of
  points on the boundary for which the radial limits exist is preserved by a conformal map.
 \end{remark}

 The conformal invariance of $\mathcal{D}_{\mathrm{harm}}(\Omega)$ yields in particular its
 invariance under biholomorphisms $T: \Omega \rightarrow \Omega$.
 Furthermore, since sets of harmonic measure $0$ are also conformally invariant (that is, independent of the point $p$), it immediately
 follows that the possible set of functions on the boundary (defined up to sets of harmonic measure $0$) is independent of the choice
 of $p \in \Omega$. Thus we can make the following definition.
 \begin{definition}  \label{de:Osborn_space}
  For a Jordan curve $\Gamma$ bounding a domain $\Omega$, let $\mathcal{H}_\pm(\Gamma)$ denote
  the set of functions obtained as boundary values of functions in $\mathcal{D}_{\mathrm{harm}}(\Omega^\pm)$.
 \end{definition}
 It is important to observe that $\mathcal{H}_\pm(\Gamma)$ are quite large classes of functions.
 For example, $\mathcal{H}_+(\Gamma)$ contains $H \circ f^{-1}$ for all $H$ which are $C^1$ on
 the closed disk $\mathrm{cl}\disk^+$ where $f:\disk^+ \rightarrow \Omega^+$ is a conformal bijection.

 Note that harmonic measure on the inside and outside might be incomparable, and
 thus we must carry the distinction between boundary values obtained from the inside
 and outside.  On the other hand, by applying the reflection $\mathfrak{R}(\mathbb{S}^1)$,
 which preserves harmonic measure with respect to $0$ and $\infty$, we have that
 \[  \mathcal{H}_\pm(\mathbb{S}^1) = \mathcal{H}(\mathbb{S}^1). \]
 \begin{remark}  The existence of boundary
  values can also be viewed in the following way.  The limit exists on
  the ideal boundary of $\Omega$.  So long as we consider only domains bounded
  by Jordan curves, then by the Osgood-Carath\'eodory theorem, there is a one-to-one
  correspondence between points on
  the ideal boundary and the boundary $\partial \Omega$ in $\mathbb{C}$.
 \end{remark}

 Theorem \ref{th:boundary_values_disk} has the following immediate consequence:
 \begin{theorem}  Let $\Gamma$ be a Jordan curve in $\sphere$ bounding domains $\Omega^\pm$.
 Any $h \in \mathcal{H}_\pm(\Gamma)$ is the boundary values of a unique $H \in \mathcal{D}_{\mathrm{harm}}(\Omega^\pm)$.
 \end{theorem}

 Finally, we conclude with a result which will be necessary in the next section.  It is a consequence of
 a theorem due to N. Arcozzi and R. Rochberg \cite{Arcozzi_Rochberg}, which states that if $\phi:\mathbb{S}^1 \rightarrow \mathbb{S}^1$ is a quasisymmetry and $I$ is a closed subset of $\mathbb{S}^1$, then there is a constant $C>0$ depending
  only on $\phi$ such that $\frac{1}{C}\, \text{cap}(I) \leq \text{cap}(\phi(I))\leq C \, \text{cap}(I).$
 In \cite{Arcozzi_Rochberg} this result is stated to be a corollary of a theorem due to
 A. Beurling and L. Ahlfors \cite{BeurlingAhlfors_boundary}, and an independent proof is also given.
 \begin{theorem}   \label{th:outer_capacity_preserved}
  If $\phi:\mathbb{S}^1 \rightarrow \mathbb{S}^1$ is a quasisymmetry then $\phi$ takes
  Borel sets of capacity zero to Borel sets of capacity zero.
 \end{theorem}
 \begin{proof}First we observe that, since $\phi$ and $\phi^{-1}$ are homeomorphisms, they take Borel sets to Borel sets. Denoting the capacity of a set $A$ by $\text{cap}(A),$
  let $E$ be a Borel subset of $\mathbb{S}^1$  with $\text{cap}(E)=0.$\\
  Assume, by way of contradiction, that $\text{cap}(\phi(E)) >0$.  Then since
  the capacity of $\phi(E)$ is the supremum of the capacities of all compact subsets $K$ of $\phi(E)$ (see \cite[Theorem 5.1.2(b)]{Ransford_book}), there
  is a compact $K \subseteq \phi(E)$ such that $\text{cap}(K) >0$. Moreover, $\phi^{-1}(K)$ is closed (in fact compact) in $\mathbb{S}^1$, since $\phi^{-1}$ is also a quasisymmetry. Therefore
   the result of Arcozzi and Rochberg mentioned above yields that there is a constant $C>0$ depending
  only on $\phi^{-1}$ such that $\text{cap}(\phi^{-1}(K))\geq \frac{1}{C} \, \text{cap}(K) >0$.
  But since $\phi^{-1}(K)\subseteq E$, Theorem 5.1.2(b) in \cite{Ransford_book} would then imply that $\text{cap}(E)>0$, which is a
  contradiction. Thus $\text{cap}(\phi(E))=0$.
 \end{proof}
\end{subsection}
\begin{subsection}{Harmonic reflection in quasicircles}

 We require the following results characterizing quasisymmetries.  For
 $h \in \mathcal{H}_\pm(\Gamma)$ and a homeomorphism $\phi:\mathbb{S}^1 \rightarrow \mathbb{S}^1$,
 define
  \[  C_\phi h = h \circ \phi.  \]
 If we furthermore assume that $h \circ \phi \in L^1(\mathbb{S}^1)$ then we can define
 \[  \hat{C}_\phi h = h \circ \phi - \frac{1}{2\pi} \int_{\mathbb{S}^1} h \circ \phi(e^{i\theta}) d\theta.  \]

Following S. Nag and D. Sullivan, we define
 \[  \dot{\mathcal{H}}(\mathbb{S}^1) = \left\{ h \in \mathcal{H}(\mathbb{S}^1) \,:\, \hat{h}(0)= 0  \right\}.  \]

We observe that the restriction of the seminorm $\| \cdot \|_{\mathcal{H}(\mathbb{S}^1)}$
to $\dot{\mathcal{H}}(\mathbb{S}^1)$ is a norm.

Moreover, note that the restriction and extension are isometries with respect to $\| \cdot \|_{{\mathcal{H}}(\mathbb{S}^1)}$
  and $\| \cdot \|_{\mathcal{D}_{\mathrm{harm}}(\disk^+)}$, if we assume that the extension is zero at $0$.  We then have the following theorem of Nag-Sullivan \cite{NagSullivan}.
 \begin{theorem}[\cite{NagSullivan}, Theorem 3.1]  \label{th:Nag_Sullivan}
  Let $\phi:\mathbb{S}^1 \rightarrow \mathbb{S}^1$ be a homeomorphism.  Then $\phi$ is a
  quasisymmetry if and only if $h \circ \phi \in \mathcal{H}(\mathbb{S}^1)$
  for all $h \in \mathcal{H}(\mathbb{S}^1)$ and $\hat{C}_\phi:\dot{\mathcal{H}}(\mathbb{S}^1)
   \rightarrow \dot{\mathcal{H}}(\mathbb{S}^1)$ is bounded with respect to
   the norm $\| \cdot \|_{{\mathcal{H}}(\mathbb{S}^1)}$.
 \end{theorem}
 \begin{remark}  The statement of the Theorem 3.1 in \cite{NagSullivan} omitted
 the boundedness condition in one direction, which was used in their proof. We have thus re-worded the statement of their theorem slightly.
 \end{remark}
 \begin{remark} \label{re:Nag_Sullivan_equivalent}  An equivalent way to state the theorem
  is that a homeomorphism  $\phi$ is a
  quasisymmetry if and only if $h \circ \phi \in \mathcal{H}(\mathbb{S}^1)$
  for all $h \in \mathcal{H}(\mathbb{S}^1)$ and $C_\phi$ is bounded with respect to
   the semi-norm $\| \cdot \|_{{\mathcal{H}}(\mathbb{S}^1)}$.
 \end{remark}

 \begin{theorem} \label{th:quasisymmetry_composition}
  Let $\phi:\mathbb{S}^1 \rightarrow \mathbb{S}^1$ be a homeomorphism.
  Then $\phi$ is a quasisymmetry if and only if $C_\phi$ takes $\mathcal{H}(\mathbb{S}^1)$ into $\mathcal{H}(\mathbb{S}^1)$
  and is bounded with respect to $\| \cdot \|_{\mathcal{H}_0(\mathbb{S}^1)}$.
 \end{theorem}
 \begin{proof}  The fact that $\phi$ a quasisymmetry implies that $C_\phi$ is a bounded map
  into $\mathcal{H}(\mathbb{S}^1)$ was proven by the authors in
  \cite[Theorem 2.7]{SchippersStaubach_Proc}.

  To prove the converse, assume that $C_\phi$ is bounded from
  $\mathcal{H}(\mathbb{S}^1)$ into $\mathcal{H}(\mathbb{S}^1)$ with respect to $\| \cdot \|_{\mathcal{H}_0(\mathbb{S}^1)}$.
  Since $\mathcal{H}(\mathbb{S}^1) \subset L^2(\mathbb{S}^1) \subset L^1(\mathbb{S}^1)$
  the average
  \[  \frac{1}{2\pi} \int_{\mathbb{S}^1} h \circ \phi(e^{i\theta}) d\theta  \]
  is defined for any $h \in \mathcal{H}(\mathbb{S}^1)$.
  Thus for any $h\in \dot{\mathcal{H}}(\mathbb{S}^1)$ we can meaningfully make the
  following estimate:
  \begin{align*}
   \| \hat{C}_\phi h \|_{{\mathcal{H}}(\mathbb{S}^1)}
         & =  \| C_\phi h  \|_{{\mathcal{H}}(\mathbb{S}^1)}
         \leq \| C_\phi h  \|_{\mathcal{H}_0(\mathbb{S}^1)} \\
         & \leq M \| h   \|_{\mathcal{H}_0(\mathbb{S}^1)} = M \| h   \|_{\mathcal{H}(\mathbb{S}^1)}.
  \end{align*}
  Thus the claim follows from Theorem \ref{th:Nag_Sullivan}.
 \end{proof}
 Theorem \ref{th:quasisymmetry_composition} is a strengthening of Theorem \ref{th:Nag_Sullivan}
 in one direction and a weakening in the other.

 This leads to the following interesting characterization of quasicircles. Let
  $\mathcal{D}_{\mathrm{harm}}(\Omega^+)_{r}$ temporarily denote the set of $h \in \mathcal{D}_{\mathrm{harm}}(\Omega^+)$
  whose boundary values are equal to an element of $\mathcal{H}_-(\Gamma)$ almost everywhere
  with respect to harmonic measure induced by $\Omega^-$.  We have a well-defined
  ``reflection'' $\mathfrak{R}^{+-}(\Gamma): \mathcal{D}_{\mathrm{harm}}(\Omega^+)_{r} \rightarrow \mathcal{D}_{\mathrm{harm}}(\Omega^-)$
  given by letting $\mathfrak{R}^{+-}(\Gamma) h$ be the unique harmonic function in $\mathcal{D}_{\mathrm{harm}}(\Omega^-)$ whose boundary
  values equal those of $h$ almost everywhere.  We can also define a reflection $\mathfrak{R}^{-+}(\Gamma)$ similarly.
 \begin{theorem} \label{th:quasicircle_characterization}
  Let $\Gamma$ be a Jordan curve and let $\Omega^\pm$ denote
  the complements of $\Gamma$ in $\sphere$.  The following are equivalent:
  \begin{enumerate}
   \item for all $h \in \mathcal{D}_{\mathrm{harm}}(\Omega^+)$, the boundary values of $h$ are in $\mathcal{H}_-(\Gamma)$,
    and $\mathfrak{R}^{+-}(\Gamma):\mathcal{D}_{\mathrm{harm}}(\Omega^+) \rightarrow \mathcal{D}_{\mathrm{harm}}(\Omega^-)$ is a bounded
    operator;
   \item for all $h \in \mathcal{D}_{\mathrm{harm}}(\Omega^-)$, the boundary values of $h$ are in $\mathcal{H}_+(\Gamma)$,
    and $\mathfrak{R}^{-+}(\Gamma):\mathcal{D}_{\mathrm{harm}}(\Omega^-) \rightarrow \mathcal{D}_{\mathrm{harm}}(\Omega^+)$ is a bounded
    operator;
   \item  $\Gamma$ is a quasicircle.
  \end{enumerate}

 Finally, if $\Gamma$ is a quasicircle, then for any $h \in \mathcal{D}_{\text{harm}}(\Omega^+)$,
  the boundary values of $\mathfrak{R}^{+-}(\Gamma) h$ exist and agree with those of $h$, except on
  a set $K$, which is simultaneously the image of a subset of $\mathbb{S}^1$ of  capacity zero under a conformal map $f:\disk^+ \rightarrow \Omega^+$,
  and the image of a subset of $\mathbb{S}^1$ of capacity zero under
  a conformal map $g:\disk^- \rightarrow \Omega^-$.
  In particular, $K$ is a set of harmonic measure zero with respect to both $\Omega^+$ and
  $\Omega^-$.

  The same claim holds for $h \in \mathcal{D}_{\text{harm}}(\Omega^-)$ and $\mathfrak{R}^{-+}(\Gamma)h$.
 \end{theorem}
 \begin{proof}
  First, observe that all three conditions hold for $\Gamma= \mathbb{S}^1$ by Theorem
  \ref{th:boundary_values_disk} and Remark \ref{re:equal_on_Dplusminus}.  We will
  use this repeatedly in what follows.

  Assume that $\Gamma$ is a quasicircle.   Let $f:\disk^+ \rightarrow \Omega^+$
  and $g:\disk^- \rightarrow \Omega^-$ be conformal maps, which must have quasiconformal extensions
  to $\sphere$.  Setting $\phi = g^{-1} \circ f$, we have that $\phi^{-1}$ is a quasisymmetry.
  Now let $h \in \mathcal{H}_+(\Gamma)$ and let $H_+$ be its harmonic extension  in $\mathcal{D}_{\text{harm}}(\Omega^+)$.
  We will show that $h$ has an extension $H_- \in \mathcal{D}_{\text{harm}}(\Omega^-)$ and that the map $H_+ \mapsto H_-$ is bounded from
 $\mathcal{D}_{\mathrm{harm}}(\Omega^+)$ to $\mathcal{D}_{\mathrm{harm}}(\Omega^-)$.

 Denote the extension from $\mathcal{H}(\mathbb{S}^1)$
  to $\mathcal{D}_{\mathrm{harm}}(\disk^\pm)$ by $e_\pm$ and the map from $\mathcal{D}_{\mathrm{harm}}(\disk^\pm)$ to the boundary values
   in $\mathcal{H}(\mathbb{S}^1)$ by $r_\pm$.  We have that $r_\pm$ and $e_\pm$ are bounded maps
   with respect to the ${\mathcal{H}}(\mathbb{S}^1)$ norm.
  Denote also
  \begin{align*}
   C_f: \mathcal{D}_{\mathrm{harm}}(\Omega^+) & \rightarrow \mathcal{D}_{\mathrm{harm}}(\disk^+) \\
    h & \mapsto h \circ f
  \end{align*}
  and
  \begin{align*}
   C_{g^{-1}}: \mathcal{D}_{\mathrm{harm}}(\disk^-) & \rightarrow \mathcal{D}_{\mathrm{harm}}(\Omega^-) \\
   h \mapsto h \circ g^{-1}
  \end{align*}
  which preserve the semi-norm by change of variables.
  By Theorem \ref{th:quasisymmetry_composition}, we then have that
  $C_{g^{-1}} e_- C_{\phi^{-1}} r_+ C_f H_+$
  is in $\mathcal{D}_{\mathrm{harm}}(\Omega^-)$ and the map $C_{g^{-1}} e_- C_{\phi^{-1}} r_+ C_f $ is bounded.

Now set $H_- = C_{g^{-1}} e_- C_{\phi^{-1}} r_+ C_f H_+$; it
  remains to be shown that the boundary values of $H_-$ and $h$ agree except on a
  set $K$ with the specified properties.

  First observe that the boundary values $\widetilde{H_+ \circ f}$ of $H_+ \circ f \in \mathcal{D}_{\text{harm}}(\disk^+)$
  exist in $\mathcal{H}(\mathbb{S}^1)$ in the sense of Osborn, except on a
  Borel set $E$ of capacity zero, by Theorem \ref{th:boundary_values_disk} and
  Remark \ref{re:outer_cap_contained}.
 By Theorem \ref{th:outer_capacity_preserved}, $\phi(E)$ also has capacity zero, since $\phi$ is a quasisymmetry. Now consider $\widetilde{H_+ \circ f} \circ \phi^{-1}$.  By
  Theorem \ref{th:quasisymmetry_composition}, this is in $\mathcal{H}(\mathbb{S}^1)$,
  and by Theorem \ref{th:boundary_values_disk} it has an extension to a function $G \in \mathcal{D}_{\text{harm}}(\disk^-)$ (say) whose boundary values
  agree with $\widetilde{H_+ \circ f} \circ \phi^{-1}$
  except on a Borel set $F$ of capacity zero in $\mathbb{S}^1$.
  If we set $I= F \cup \phi(E)$, then we have that $I$ has zero capacity. To see this, we observe that since both $F$ and $\phi(E)$ have capacity zero and are both bounded Borel sets, then their outer capacity which is by Choquet's theorem \cite{Choquet} equal to their capacity, is also equal to zero. Therefore using the subadditivity of the outer capacity under countable unions \cite[Theorem 2.1.9]{El-Fallah_etal_primer} and Choquet's theorem, we have that $\text{cap}( I)=0$.


  Now if we set $K=g(I)$, then $K$ is the image of a set of
   capacity zero under $g$, and the function $G \circ g^{-1} = H_-$ is in
  $\mathcal{D}_{\text{harm}}(\Omega^-)$.  Furthermore the boundary values agree
  with $\widetilde{H_+ \circ f} \circ \phi^{-1} \circ g^{-1} = H_+$ on $\Gamma \backslash g(I) = \Gamma \backslash K$
  (see Remark \ref{re:conformal_invariance_limits}).
   On the other hand, $K=g(I) = g(F) \cup g(\phi(E)) = f(\phi^{-1}(F)) \cup
  f(E)= f(\phi^{-1}(F) \cup E)$.  Since $\phi^{-1}$ is also a quasisymmetry, by
  the argument applied above $I'=\phi^{-1}(F) \cup E$ has capacity zero and hence
  $K = f(I')$ is the image of a subset of $\mathbb{S}^1$ of capacity zero under $f$ as claimed.

  As we mentioned in the proof of Theorem \ref{Osborns sats}, subsets of $\mathbb{S}^1$ of logarithmic capacity zero are null-sets with respect to the harmonic measure. Therefore  \cite[I.3, equation (3.3)]{Garnett_Marshall} yields that $K=g(I)$ has harmonic measure zero with respect to
  $\Omega^-$ and since $K=f(I')$, it also has harmonic measure zero with respect
  to $\Omega^+$.
  Thus we have shown that (3) implies (1), and also that (3) implies the final claim. The proof that (3) implies (2) and the final claim is similar, hence omitted.  Thus, once
  the equivalence of (1), (2), and (3) is demonstrated, the proof of the theorem
  will be complete.

  We show that (1) implies (3); the proof that (2) implies (3) is similar.   Assuming (1), we
  define the bounded reflection $\mathfrak{R}^{+-}(\Gamma): \mathcal{D}_{\mathrm{harm}}(\Omega^+) \rightarrow \mathcal{D}_{\mathrm{harm}}(\Omega^-)$.  Let $f$ and $g$ be conformal maps of $\disk^+$
  and $\disk^-$ onto $\Omega^+$ and $\Omega^-$ respectively. By the Osgood-Carath\'eodory theorem,
  $f$ and $g$ extend to homeomorphisms of $\mathbb{S}^1$ to $\Gamma$.  Thus we may define
  a homeomorphism $\phi:\mathbb{S}^1 \rightarrow \mathbb{S}^1$ by $\phi = g^{-1} \circ f$.
  For any $H_+ \in \mathcal{D}_{\mathrm{harm}}(\disk^+)$, $C_g \mathfrak{R}^{+-}(\Gamma) C_{f^{-1}} H_+$ is in $\mathcal{D}_{\mathrm{harm}}(\disk^-)$.
  Furthermore, for any $h \in \mathcal{H}(\mathbb{S}^1)$, by comparing boundary values as above,
  $C_{\phi^{-1}} h = r_- C_g
  \mathfrak{R}^{+-}(\Gamma) C_{f^{-1}} e_+ h$.  Thus
  \[  C_{\phi^{-1}} = r_- C_g  \mathfrak{R}^{+-}(\Gamma) C_{f^{-1}} e_+  \]
  and since all maps on the right hand side are bounded, we conclude that $C_{\phi^{-1}}$ is
  a bounded operator on $\mathcal{H}(\mathbb{S}^1)$ with respect to the semi-norm. So by
  Theorem \ref{th:Nag_Sullivan} (see Remark \ref{re:Nag_Sullivan_equivalent})
  we have that $\phi$ is a quasisymmetry.

  This implies that $g$ and $f$ have quasiconformal
  extensions to $\sphere$.  This is a consequence of the proof of the conformal welding theorem, but
  not the statement, so we supply the argument.  We will use the following fact:
  given a quasicircle $\gamma \subset \sphere$, a
  continuous map $\Phi$ of $\sphere$ which is quasiconformal on the complements of $\gamma$ is
  quasiconformal on $\sphere$.

  Let $w_\mu:\disk^- \rightarrow \disk^-$ be a quasiconformal extension of $\phi$, which exists
  by the Beurling-Ahlfors extension theorem.  Let $\mu$ be the Beltrami differential of $w_\mu$.
  Let $w^\mu:\sphere \rightarrow \sphere$ be a solution to the Beltrami equation with
  Beltrami differential $\mu$ on $\disk^-$ and $0$ on $\disk^+$.  Define the map
  $\Phi: \sphere \rightarrow \sphere$ to be the continuous extension of
  \[  \Phi(z) = \left\{ \begin{array}{cc} g \circ w_\mu \circ (w^\mu)^{-1}(z) & z \in w^\mu(\disk^-) \\
       f \circ (w^\mu)^{-1}(z) & z \in w^\mu(\disk^+).  \end{array}  \right.  \]
  A continuous extension exists, since $(w^\mu)^{-1}$ is continuous
  on the omitted set $\gamma = w^\mu(\mathbb{S}^1)$ (which is a quasicircle), and $g \circ w_\mu = f$
  on $\mathbb{S}^1$ by definition of $w_\mu$.  Thus $\Phi$ is quasiconformal, and since $\gamma$ is a
  quasicircle, $\Gamma = \Phi(\gamma)$ is also a quasicircle.  This completes the proof.
 \end{proof}

 \begin{remark}  Note that we do not claim that the quasisymmetry $\phi$ in
 the proof takes sets of harmonic measure zero to sets of harmonic measure zero.
 We are grateful to the referee for drawing our attention to this subtlety.
 \end{remark}

 The following immediate consequence deserves to be singled out, since it will allow us to
 consistently define the set of boundary values of harmonic functions of finite Dirichlet energy.
 \begin{corollary}
  Let $\Gamma \subset \sphere$ be a quasicircle.  Then $\mathcal{H}_+(\Gamma) = \mathcal{H}_-(\Gamma)$.
 \end{corollary}
 That is, $\mathcal{D}_{\mathrm{harm}}(\Omega^+)$ and $\mathcal{D}_{\mathrm{harm}}(\Omega^-)$
 have the same boundary values.

 Theorem \ref{th:quasicircle_characterization} also holds for the pointed norms.
 \begin{theorem}  Let $\Gamma$ be a closed Jordan curve bounding $\Omega^\pm$.
   For any fixed $p^\pm \in \Omega^\pm$, the maps
    $\mathfrak{R}^{+-}(\Gamma):\mathcal{D}_{\mathrm{harm}, p^{+}}(\Omega^+) \rightarrow \mathcal{D}_{\mathrm{harm},  p^{-}}(\Omega^-)$
     and $\mathfrak{R}^{-+}(\Gamma):\mathcal{D}_{\mathrm{harm}, p^{-}}(\Omega^-) \rightarrow \mathcal{D}_{\mathrm{harm},p^{+}}(\Omega^+)$
     are bounded if and only if $\Gamma$ is a quasicircle.
 \end{theorem}
 \begin{proof}
  Choose $f:\disk^+ \rightarrow \Omega^+$ and $g:\disk^- \rightarrow \Omega^-$ to
  be conformal maps such that $f(0)=p^+$ and $g(\infty)=p^-$.  We define extension
  and restriction operators $r_\pm$ and $e_\pm$ as in
  the proof of Theorem \ref{th:quasicircle_characterization} except that now we observe
  that they are isometries between $\mathcal{D}_{\mathrm{harm}}(\disk^\pm)$ and $\mathcal{H}(\mathbb{S}^1)$
  with respect to the $\| \cdot \|_{\mathcal{H}_0(\mathbb{S}^1)}$ and $\| \cdot \|_{\mathcal{D}_{\mathrm{harm},q^{\pm}}(\disk^\pm)}$
  norms, where $q^+=0$ and $q^-=\infty$.
  The proof proceeds as in Theorem \ref{th:quasicircle_characterization}.\\

 Conversely for the welding map  $\phi = g^{-1} \circ f$ the proof of Theorem \ref{th:quasicircle_characterization} shows that that $C_{\phi^{-1}} = r_- C_g
  \mathfrak{R}^{+-}(\Gamma) C_{f^{-1}} e_+ ,$ so the boundedness of $\mathfrak{R}^{+-}(\Gamma)$ will yield the boundedness of $C_{\phi^{-1}}$ on $\mathcal{H}(\mathbb{S}^1)$, which by Theorem \ref{th:quasisymmetry_composition} implies that $\phi^{-1}$ (and therefore also $\phi$) is a quasisymmetry.  Repeating the conformal welding argument in the proof of Theorem \ref{th:quasicircle_characterization} we have that
  $\Gamma$ is a quasicircle. Following a similar argument, the boundedness of  $\mathfrak{R}^{-+}(\Gamma)$ also
   implies that $\Gamma$ is a quasicircle.
 \end{proof}

 Theorem \ref{th:quasicircle_characterization} has the following important
 consequence.  For any quasicircle $\Gamma$ bounding domains $\Omega^\pm$, there is a constant $C>0$
 such that for any $h \in \mathcal{H}(\Gamma)$
 the extensions $H_\pm \in \mathcal{D}_{\mathrm{harm}}(\Omega^\pm)$ satisfy
 \begin{equation} \label{eq:norm_equivalent_in_out}
  \frac{1}{C} \| H_- \|_{\mathcal{D}_{\mathrm{harm}}(\Omega^-)} \leq  \| H_+ \|_{\mathcal{D}_{\mathrm{harm}}(\Omega^+)} \leq C \| H_- \|_{\mathcal{D}_{\mathrm{harm}}(\Omega^-)}.
 \end{equation}
 Thus we may make the following definition:
 \begin{definition}  Let $\Gamma$ be a quasicircle in $\sphere$.
  Define the ``Douglas-Osborn space'' by
  $\mathcal{H}(\Gamma) = \mathcal{H}_+(\Gamma) = \mathcal{H}_-(\Gamma)$.
 \end{definition}
 \begin{remark} \label{re:in_out_automatic}
  Thus, for a quasicircle, any statement regarding boundedness with respect to either $\mathcal{H}_+(\Gamma)$ or $\mathcal{H}_-(\Gamma)$
  is automatically true for both semi-norms.
 \end{remark}

 It is a remarkable property of quasicircles that the spaces are the same and the norms are equivalent.

 We note that neither Douglas nor Osborn made special mention of quasicircles.  However, we chose
 the name because of their pioneering work on expressing the Dirichlet norm in terms of boundary values.

 We can also define a pointed norm on the Douglas-Osborn space as follows.
 \begin{definition}\label{defn: Pointed Osborn equals pointec Dirichlet} Let $\Gamma$ be a quasicircle in $\sphere$.
  For $p \notin \Gamma$, define the pointed Douglas-Osborn norm as follows.  Let $\Omega$
  be the component of the complement of $\Gamma$ containing $p$.  For $h \in \mathcal{H}(\Gamma)$
  let $H$ be the unique harmonic extension in $\mathcal{D}_{\mathrm{harm}}(\Omega)$.  Define
  \[  \| h \|_{\mathcal{H}_p(\Gamma)} = \| H \|_{\mathcal{D}_{\mathrm{harm},p}(\Omega)}.       \]
 \end{definition}
 Of course, this choice is not canonical and breaks the symmetry between inside and outside
 domain.  On the other hand, this defines an actual norm as opposed to a semi-norm.

 We also immediately have that the Douglas-Osborn semi-norm is conformally invariant in the following
 sense.   If $\Gamma_1$ and $\Gamma_2$ are quasicircles bounding domains $\Omega_1$ and $\Omega_2$
 respectively, then a conformal map $f:\Omega_1 \rightarrow \Omega_2$ has a unique homeomorphic
 extension to $\Gamma_1$ and $\Gamma_2$.  We therefore have a well-defined composition
 map
 \begin{align*}
  C_f: \mathcal{H}(\Gamma_2) & \rightarrow \mathcal{H}(\Gamma_1) \\
   h & \mapsto h \circ f.
 \end{align*}
 Then we have the following immediate consequence of our definitions.
 \begin{theorem} \label{th:normal_comp_isometry} Let $\Gamma_1$ and $\Gamma_2$ be Jordan curves in $\sphere$, and let $\Omega_1^\pm$
  and $\Omega_2^\pm$ be components of their complements respectively.   If $f^\pm:\Omega_1^\pm \rightarrow \Omega_2^\pm$
  is a conformal map, then $C_{f^\pm}:\mathcal{H}(\Gamma_2) \rightarrow \mathcal{H}(\Gamma_1)$ is
  an isometry with respect to $\| \cdot \|_{\mathcal{H}_\pm(\Gamma_i)}$.

  If $p_1^\pm \in \Omega_1^\pm$ and $p_2^\pm=f(p_1^\pm) \in \Omega_2^\pm$, then $C_{f^\pm}$ is an isometry
  with respect to $\| \cdot \|_{\mathcal{H}_{p_i^\pm}(\Gamma_i)}$.
 \end{theorem}
\end{subsection}
\end{section}
\begin{section}{Jump decomposition and well-posedness of a Riemann-Hilbert problem} \label{se:jump}
\begin{subsection}{Cauchy-type operators on quasidisks}
 In this section we define a limiting Cauchy integral, and
 show that it is a bounded map from $\mathcal{H}(\Gamma)$ to the
 holomorphic Dirichlet spaces of the complement.
 By holomorphic Dirichlet spaces we mean, for a Jordan domain $\Omega \subset \sphere$,
 \[  \mathcal{D}(\Omega) = \{ h \in \mathcal{D}_{\mathrm{harm}}(\Omega) \,:\, h \ \text{is holomorphic} \}.  \]
 The semi-norm and norm on the harmonic Dirichlet space restrict to $\mathcal{D}(\Omega)$; we denote these
 by
 \[  \| h \|_{\mathcal{D}(\Omega)}^2 = \iint_{\Omega} |h'|^2 \,dA  \]
 and
 \[  \| h \|_{\mathcal{D}_p(\Omega)}^2 = |h(p)|^2 +  \iint_{\Omega} |h'|^2 \,dA.  \]

 For convenience in our Cauchy-type integral operators, throughout the rest of the paper we will assume
 that $\Gamma$ does not contain $\infty$.
 Recall that according to our conventions, $\Omega^-$ will be
 the unbounded component of the complement of $\Gamma$.  For the unbounded domain,
 we will define
 \[ \mathcal{D}_*(\Omega^-) = \{ h \in \mathcal{D}(\Omega^-) \,:\, h(\infty) = 0  \}.  \]
 Of course the restriction of $\| \cdot \|_{\mathcal{D}(\Omega^-)}$ is a norm on $\mathcal{D}_*(\Omega^-)$.

 We define an operator which will play the role of the Cauchy integral in the Riemann-Hilbert
 boundary value problem.  We will see in Theorem \ref{th:contour_integral_jump} ahead that
 it equals a Cauchy integral in a certain sense. It is important to note that in what follows, we will make statements and claims that are valid for general Jordan curves. We shall therefore emphasize that fact by including that assumption in all the statements below, until otherwise stated explicitly.

 \begin{definition} Let $\Gamma$ be a Jordan curve not containing $\infty$, bounding $\Omega^\pm$.
 We define the jump operator on $\mathcal{H}_+ (\Gamma)$ as follows.  For $h \in \mathcal{H}_+(\Gamma)$
 let $h_{\Omega^+} \in \mathcal{D}_{\text{harm}}(\Omega^+)$ be its unique extension.  Then
  \begin{equation}\label{defn: Cauchy integral}
J(\Gamma)h(z) := h_{\Omega^+}(z) \chi_{\Omega^{+}}(z)+\frac{1}{\pi }\iint_{\Omega^{+}}
   \frac{\overline{\partial} h_{\Omega^+}(\zeta)}{\zeta -z}\, dA(\zeta), \,\,\, z\in \mathbb{C}\setminus \Gamma
\end{equation}
where $\chi_{\Omega^{+}}$ denotes the characteristic function of the closure of $\Omega^{+}.$
\end{definition}
Denoting by $|\Omega^+|$ the area of $\Omega^+$, which is finite because $\Omega^+$ is bounded, we have the estimate
\begin{equation}
\iint_{\Omega^{+}} |\overline{\partial} h_{\Omega^+}(\zeta)|\, dA(\zeta)\leq |\Omega^{+}|^{\frac{1}{2}} \Big{(} \iint_{\Omega^{+}} |\overline{\partial} h_{\Omega^+}(\zeta)|^2 \, dA(\zeta)\Big{)}^{\frac{1}{2}} \leq |\Omega^{+}|^{\frac{1}{2}} \Vert h_{\Omega^+} \Vert_{\mathcal{D}_{\text{harm}}(\Omega^+)}
\end{equation}
so $\overline{\partial} h_{\Omega^+}\in L^1 (\Omega^{+}).$
It is also easily seen that the formula implies that $\lim_{z\to \infty} J(\Gamma)h (z) =0$, and it is understood that
$J(\Gamma)h(z)$ extends to a function on $\Omega^+ \sqcup \Omega^-$.

To prove that $J(\Gamma)$ is bounded, we also need estimates for a certain integral operator.  Initially, for $\varphi \in C_{0}^\infty (\Omega^+)$ ($C^\infty $ with compact support in $\Omega^+$) we define the operator $T_{\Omega^+}$ via

\begin{equation}\label{defn: operator T}
 T_{\Omega^+}\varphi(z)= \frac{1}{\pi}\iint_{\Omega^+} \frac{\varphi(\zeta)}{\zeta-z}\, dA(\zeta).
\end{equation}
One has that

\begin{equation}\label{estim: L2 norm of operator T}
 \Vert T_{\Omega^+}\varphi\Vert_{L^2(\Omega^+)}\leq C \Vert \varphi\Vert_{L^2 (\Omega^+)}
\end{equation}
and
\begin{equation}\label{estim: Sobolev norm of operator T}
\Vert \partial (T_{\Omega^+}\varphi) \Vert_{L^2 (\Omega^\pm)} +\Vert \overline{\partial}( T_{\Omega^+}\varphi )\Vert_{L^2 (\Omega^\pm)}\leq C \Vert \varphi\Vert_{L^2 (\Omega^+)}.
\end{equation}
 The estimate \eqref{estim: L2 norm of operator T} was proven in \cite[Lemma 3.1]{RSS_WPjump}.
Estimate \eqref{estim: Sobolev norm of operator T} is a direct consequence of the facts that $\partial(T_{\Omega^+}\varphi)(z)= \mathrm{P.V.}\iint_{\Omega^+} \frac{\varphi(\zeta)}{(\zeta-z)^2}\, dA(\zeta),$ which is the Beurling transform of $\varphi \chi_{\Omega^+}$, and $\overline{\partial} (T_{\Omega^+}\varphi)(z)=-\varphi(z)\chi_{\Omega^+}(z).$ Both derivatives here are in the sense of distributions. Now since the Beurling transform $B \varphi(z):=  \mathrm{P.V.}\iint_{\mathbb{C}} \frac{\varphi(\zeta)}{(\zeta-z)^2}\, dA(\zeta)$ is a Calder\'on-Zygmund singular integral operator, it is well-known that $B$ is bounded on $L^p (\mathbb{C})$ for $1<p<\infty$. See e.g. \cite{Lehto} for the proof of these facts.

Now we can state and prove the following key result.
 \begin{theorem} \label{th:J_bounded}  Let $\Gamma$ be a Jordan curve not containing $\infty$, bounding
  $\Omega^\pm$. Then $J(\Gamma) h$ is holomorphic in $\Omega^+ \sqcup \Omega^-$
  for all $h \in \mathcal{H}_{+}(\Gamma)$.  Furthermore the operators
  \begin{align*}
  P(\Omega^+):\mathcal{H}_{+}(\Gamma) & \rightarrow \mathcal{D}(\Omega^+) \\
  h & \mapsto \left. J(\Gamma) h \right|_{\Omega^+}
 \end{align*}
 and
 \begin{align*}
  P(\Omega^-):\mathcal{H}_{+}(\Gamma) & \rightarrow \mathcal{D}(\Omega^-) \\
  h & \mapsto \left. -J(\Gamma) h \right|_{\Omega^-}
 \end{align*} are bounded with respect to the semi-norms $\| \cdot \|_{\mathcal{H}_{\color{blue}{+}}(\Gamma)}$,
 $\| \cdot \|_{\mathcal{D}(\Omega^+)}$ and $\| \cdot\|_{\mathcal{D}(\Omega^-)}$.
 \end{theorem}
 \begin{proof} For $h \in \mathcal{H}_{\color{blue}{+}}(\Gamma)$, set
   $$h_{\pm}=\pm(J(\Gamma)h)|_{\Omega^{\pm}}.$$
 We first show that $h_\pm$ are holomorphic in $\Omega^{\pm}$. Since $J(\Gamma)h= h_{\Omega^+} +T_{\Omega^+} \overline{\partial}h_{\Omega^+}$ one has that $\overline{\partial}J(\Gamma)h(z)=  \overline{\partial}h_{\Omega^{+}}-\overline{\partial} h_{\Omega^{+}}=0.$ On $\Omega^-$ one has $$J(\Gamma)h(z)= \frac{1}{\pi }\iint_{\Omega^{+}} \frac{\overline{\partial}h_{\Omega^+}(\zeta)}{\zeta -z}\, dA(\zeta),$$ which is obviously holomorphic on $\Omega^- \backslash \{ \infty \}$.  From the expression  \eqref{defn: Cauchy integral} we also see that $J(\Gamma)h \rightarrow 0$ as $z \rightarrow \infty$,
so in fact $J(\Gamma) h$ is holomorphic in a neighbourhood of $\infty$.\\

Next we show that $h \mapsto h_\pm$ are bounded.
The estimate \eqref{estim: Sobolev norm of operator T} for $\Omega^\pm$ yields that
\begin{equation}\label{continious dependence on initial data plus}
              \Vert h_{+}\Vert_{\mathcal{D}(\Omega^{+})} =\Vert (J(\Gamma)h)|_{\Omega^+ }\Vert_{\mathcal{D}(\Omega^{+})}
               \leq \Vert h_{\Omega^+}   \Vert_{\mathcal{D}_{\text{harm}}(\Omega^{+})}+
                C \Vert \overline{\partial}h_{\Omega^+}  \Vert_{L^2(\Omega^{+})}
                \leq C \Vert h_{\Omega^+} \Vert_{\mathcal{D}_{\text{harm}}(\Omega^{+})}
\end{equation}
and
\begin{equation}\label{continious dependence on initial data minus}
              \Vert h_{-}\Vert_{\mathcal{D}(\Omega^{-})} =\Vert (J(\Gamma)h)|_{\Omega^- }\Vert_{\mathcal{D}(\Omega^{-})}
               \leq  C \Vert \overline{\partial}h_{\Omega^+} \Vert_{L^2(\Omega^{+})} \leq C \Vert h_{\Omega^+} \Vert_{\mathcal{D}_{\text{harm}}(\Omega^{+})}.
\end{equation}
which proves the claim.
 \end{proof}
 \begin{corollary}\label{cor:boundedness of P(omega)}  Let $\Gamma$ be a Jordan curve not containing $\infty$, bounding $\Omega^\pm$.
 Let $p \in \Omega^+$.  $P(\Omega^+)$ is bounded with
 respect to the norms $\| \cdot \|_{\mathcal{H}_p(\Gamma)}$ and $\| \cdot \|_{\mathcal{D}_p(\Omega^+)}$;
 and $P(\Omega^-)$ is bounded with respect to $\Vert\cdot\Vert_{\mathcal{H}_p(\Gamma)}$ and $\| \cdot \|_{\mathcal{D}(\Omega^-)}$.
 \end{corollary}
 \begin{proof}  The claim for $P(\Omega^-)$ follows from Theorem \ref{th:J_bounded} and
 the fact that $\| \cdot \|_{\mathcal{H}(\Gamma)} \leq  \| \cdot \|_{\mathcal{H}_p(\Gamma)}$. Therefore we shall only prove the claim for $P(\Omega^+)$. For the pointed norm we have

\begin{align*}
\Vert (J(\Gamma)h)|_{\Omega^+ }\Vert^2_{\mathcal{D}_p(\Omega^{+})} &= |J(\Gamma)h (p)|^2 +\Vert (J(\Gamma)h)|_{\Omega^+ }\Vert^2 _{\mathcal{D}(\Omega^{+})}\\ & = |h_{\Omega^+ }(p)+ T\overline{\partial} h_{\Omega^+} (p)|^2 +\Vert (J(\Gamma)h)|_{\Omega^+ }\Vert^2_{\mathcal{D}(\Omega^{+})},
\end{align*}
where the operator $T$ is defined by (\ref{defn: operator T}).

Through \eqref{continious dependence on initial data plus} we already know that $\Vert (J(\Gamma)h)|_{\Omega^+ }\Vert_{\mathcal{D}(\Omega^{+})}\leq C \Vert h_{\Omega^+} \Vert_{\mathcal{D}_{\text{harm}}(\Omega^{+})}.$

Moreover $T\overline{\partial} h_{\Omega^+}(z)$ is a harmonic function in $\Omega ^+$, if $ h_{\Omega^+}$ is harmonic in $\Omega ^+$. This is because
$ \overline{\partial} (T \overline{\partial} h_{\Omega^+}(z))=  -\overline{\partial} h_{\Omega^+} (z)$, for $z\in\Omega^+$.

Now since $p\in\Omega^+$ and $T\overline{\partial}h_{\Omega^+}$ is harmonic in $\Omega^+$, there is an $r>0$ such that $\mathbb{D}(p,r)\subset \Omega^+$ and by the mean-value theorem for harmonic functions one has

\begin{equation}\label{meanvalue theorem}
  |T\overline{\partial}h_{\Omega^+}(p)|\leq \frac{1}{\pi r^2} \iint_{\mathbb{D}(p,r)}|T\overline{\partial}h_{\Omega^+}(z)|\, dA(z).
\end{equation}

Moreover Jensen's inequality yields

\begin{equation}\label{jensen}
|T\overline{\partial}h_{\Omega^+}(p)|^2\leq \frac{|\Omega^+|}{\pi^2 r^4} \iint_{\Omega^+}|T\overline{\partial}h_{\Omega^+}(z)|^2\, dA(z).
\end{equation}

Hence, using estimates \eqref{meanvalue theorem} and \eqref{jensen}, together with the $L^2$-estimate \eqref{estim: L2 norm of operator T}, we have

\begin{equation}
| T\overline{\partial} h_{\Omega^+} (p)|^2 \leq C \Vert T\overline{\partial} h_{\Omega^+}\Vert_{L^2(\Omega^+)} ^2 \leq C \Vert \overline{\partial} h_{\Omega^+}\Vert_{L^2(\Omega^+)} ^2 \leq C  \Vert h_{\Omega^+} \Vert^2 _{\mathcal{D}_{\text{harm}}(\Omega^{+})}.
\end{equation}

Finally gathering all the estimates, we obtain

\begin{equation}
\Vert (J(\Gamma)h)|_{\Omega^+ }\Vert^2_{\mathcal{D}_p(\Omega^{+})}\leq C\big(|h_{\Omega^+ }(p)|^2 +  \Vert h_{\Omega^+} \Vert^2_{\mathcal{D}_{\text{harm}}(\Omega^{+})}\big)=C \Vert h_{\Omega^+} \Vert^2_{\mathcal{D}_{\mathrm{harm},p}(\Omega^{+})},
\end{equation}
and Definition \ref{defn: Pointed Osborn equals pointec Dirichlet} ends the proof of this corollary.
 \end{proof}
 \begin{remark}  In the notations $P(\Omega^\pm)$ and $J(\Gamma)$, $P$ of course stands for projection
  and $J$ stands for jump.
 \end{remark}

At this point we return to the realm of quasicircles. The following limiting integral expression is key to the results of the next section.
 \begin{theorem} \label{th:contour_integral_jump}  Let $\Gamma$ be a quasicircle, not containing $\infty$, bounding
  domains $\Omega^\pm$.   Let $f:\disk^+ \rightarrow \Omega^+$ and $g:\disk^- \rightarrow \Omega^-$
  be conformal maps.  Let $\gamma_r$ denote the curve $|w|=r$ traced counter-clockwise.  Then for
  all $h \in \mathcal{H}(\Gamma)$ and all $z \in \sphere \backslash \Gamma$,
  \begin{align}  \label{eq:J_defn_contour}
   J(\Gamma) h (z) & = \lim_{r \nearrow 1} \frac{1}{2 \pi i} \int_{f(\gamma_r)} \frac{h_{\Omega^+}(\zeta)}{\zeta - z} d\zeta  \\ \nonumber
   & =  \lim_{r \searrow 1} \frac{1}{2 \pi i} \int_{g(\gamma_r)} \frac{h_{\Omega^-}(\zeta)}{\zeta - z} d\zeta.
  \end{align}
 \end{theorem}
 \begin{proof}  We prove the first integral formula, which is straightforward.  For $z \in \Omega^-$, it follows directly from
 Stokes' theorem.  If $z \in \Omega^+$ we apply the
   mean value theorem for harmonic functions to the first term of \eqref{defn: Cauchy integral} and the Stokes theorem to its second term to obtain
   \begin{align*}
    J(\Gamma)h(z) & = \lim_{R \searrow 0} \frac{1}{2 \pi i} \int_{|\zeta -z|=R } \frac{h_{\Omega^+}(\zeta)}{\zeta - z} d\zeta
     +  \lim_{r \nearrow 1}\lim_{R \searrow 0}\frac{1}{\pi} \iint_{f(\{|w|<r\})\setminus\{|\zeta - z|\leq R\}}
     \frac{\overline{\partial} h_{\Omega^+}(\zeta)}{\zeta - z} dA(\zeta)  \\
      & = \lim_{R \searrow 0} \frac{1}{2 \pi i} \int_{|\zeta -z|=R } \frac{h_{\Omega^+}(\zeta)}{\zeta - z} d\zeta + \lim_{r \nearrow 1} \frac{1}{2 \pi i} \int_{f(|w|=r)} \frac{h_{\Omega^+}(\zeta)}{\zeta - z} d\zeta \\
      & -\lim_{R \searrow 0} \frac{1}{2 \pi i} \int_{|\zeta -z|=R } \frac{h_{\Omega^+}(\zeta)}{\zeta - z} d\zeta
      = \lim_{r \nearrow 1} \frac{1}{2 \pi i} \int_{f(\gamma_r)} \frac{h_{\Omega^+}(\zeta)}{\zeta - z} d\zeta
      \end{align*}
   which proves the first integral expression.

   The second integral expression requires more work, and proceeds as follows.
   We first show that the two limiting integrals are equal
   for those elements of the Dirichlet space which are smooth on the closure.  By applying boundedness of the reflection and density of this subset, the result follows.  To
     prove the result for smooth harmonic functions we require two preparatory facts.
   Let $U$ be an open set containing the quasicircle $\Gamma$.  Let $q(z)$ be a smooth function
   on $U$.  For $\epsilon >0$, let $V_\epsilon$ denote the region bounded by the analytic
   curves $f(\gamma_{1-\epsilon})$ and $g(\gamma_{1+\epsilon})$.
   Since $\Gamma$ is a quasicircle, it has Lebesgue measure zero (since quasiconformal
   maps take Lebesgue null-sets to Lebesgue null-sets, see e.g. \cite[I.3.5]{Lehto}).  Thus by
   Stokes' theorem
   \begin{equation} \label{eq:main_Stokes_temp}
    \lim_{\epsilon \rightarrow 0}  \frac{1}{2 \pi i}  \left(\int_{f(\gamma_{1-\epsilon})} \frac{q(\zeta)}{\zeta - z} d\zeta -
           \int_{g(\gamma_{1+\epsilon})} \frac{q(\zeta)}{\zeta - z} d\zeta \right)
          =  \lim_{\epsilon \rightarrow 0} \frac{1}{\pi} \int_{V_\epsilon}  \frac{\overline{\partial} q(\zeta)}{\zeta - z} dA_\zeta =0.
   \end{equation}
   Next, observe that if $Q(z)$ is continuous on $U \cap (\Omega^+)^c$ (where $\cdot ^c$ denotes complement)
   and zero on $\Gamma$ then
   \begin{equation} \label{eq:main_cont_zero_temp}
     \lim_{\epsilon \rightarrow 0} \int_{g(\gamma_{1+\epsilon})} \frac{Q(\zeta)}{\zeta - z} d\zeta =0, \,\,\,  z \in \sphere \backslash \Gamma
   \end{equation}

   Now let $W= U \cup \Omega^+$.  The claim is now easily seen to hold for harmonic functions on $\Omega^+$ which are
   smooth on $W$.
   Let $p$ be such a function, and let $h$ denote its boundary values.  Thus $h_{\Omega^+} = \left. p \right|_{\Omega^+}$.
   Furthermore we have that $h_{\Omega^-}= \mathfrak{R}^{+-} (\left. p\right|_{\Omega^+})$ has a continuous extension to $\text{cl} \Omega^-$ with
   respect to the spherical topology.  To see this, first we
   apply a translation to arrange that $\Omega^-$ does not contain $0$ in its closure.  We observe then that $h_{\Omega^-}(1/z)$ is the solution
   to the Dirichlet problem on the bounded domain $1/\Omega^-$ with continuous boundary data on $1/\Gamma$,
   so it is continuous on the closure of $1/\Omega^-$; the claim now follows from the facts that the Euclidean
   and spherical topologies are the same on bounded domains, and that $1/z$ and
   translations are continuous in the spherical topology.
   In particular, $h_{\Omega^-}$ is continuous on $V_{\epsilon_0} \cap (\Omega^+)^c \subset W$ for some $\epsilon_0 >0$,
   (where containment in $W$ is obtained by choosing $\epsilon_0$ sufficiently small).
   Thus since $p - h_{\Omega^-}$ is continuous on $V_{\epsilon_0} \cap (\Omega^+)^c$ and zero on $\Gamma$,
   by (\ref{eq:main_cont_zero_temp})
   \begin{align*}
    \lim_{\epsilon \rightarrow 0} \frac{1}{2 \pi i} \int_{g(\gamma_{1+\epsilon})} \frac{ h_{\Omega^-}(\zeta)}{\zeta - z} d\zeta
    & = \lim_{\epsilon \rightarrow 0} \frac{1}{2 \pi i} \int_{g(\gamma_{1+\epsilon})} \frac{ p(\zeta)}{\zeta - z} d\zeta \\
    & = \lim_{\epsilon \rightarrow 0} \frac{1}{2 \pi i} \int_{f(\gamma_{1-\epsilon})} \frac{ p(\zeta)}{\zeta - z} d \zeta \\
    & =  \lim_{\epsilon \rightarrow 0} \frac{1}{2 \pi i} \int_{f(\gamma_{1-\epsilon})} \frac{ h_{\Omega^+}(\zeta)}{\zeta - z} d \zeta
   \end{align*}
   where the second-to-last equality follows from (\ref{eq:main_Stokes_temp}).

   Finally, we observe that harmonic functions which are smooth on some $W$ as above are dense in $\mathcal{H}(\Gamma)$,
   since for example the set of holomorphic and anti-holomorphic polynomials is dense in $\mathcal{D}_{\mathrm{harm}}(\Omega)$.
   To prove this, one can for example apply the density of polynomials in the Bergman space of so-called Carath\'eodory
   domains \cite[v.3, Section 15]{Markushevich_book} which includes quasidisks.
   By the fact that differentiation is an isometry between the holomorphic Dirichlet and Bergman space (up to constants), it follows that
   polynomials are dense in the holomorphic Dirichlet space.
   Similarly conjugate polynomials are dense in the complex conjugate of the holomorphic
   Dirichlet space.  Since harmonic functions on a simply connected domain have a unique decomposition
   into holomorphic and antiholomorphic parts (up to constants), the set of $p(z) + \overline{q(z)}$ where
   $p$ and $q$ are polynomials is dense in the complex harmonic Dirichlet space.
   Since $\mathfrak{R}^{+-}$ is a bounded operator, this completes the proof.
 \end{proof}
  Note that the formulation of the previous theorem
  requires the fact that $h$ has an extension to both inside and outside, and the proof requires that the
  reflections $\mathfrak{R}^{+-}(\Gamma)$ and $\mathfrak{R}^{-+}(\Gamma)$ are bounded.
  Only quasicircles have both properties.

  Although the following corollary is now obvious, it is worth
  writing out explicitly.   Using notation as in Theorem \ref{th:contour_integral_jump},
  for a quasicircle $\Gamma$ not containing $\infty$
  we can define two Cauchy integral-type projections on $\mathcal{H}(\Gamma)$
  in two distinct reasonable ways using
  limiting integrals:
   \begin{align*}
   {\Pi}_+(\Omega^\pm): \mathcal{H}(\Gamma) & \rightarrow \mathcal{D}(\Omega^\pm) \\
    h  & \mapsto \lim_{r \nearrow 1} \frac{1}{2 \pi i} \int_{f(\gamma_r)} \frac{h_{\Omega^+}(\zeta)}{\zeta - z} d\zeta
   \end{align*}
   and
   \begin{align*}
   {\Pi}_-(\Omega^\pm): \mathcal{H}(\Gamma) & \rightarrow \mathcal{D}(\Omega^\pm) \\
   h & \mapsto - \lim_{r \searrow 1} \frac{1}{2 \pi i} \int_{g(\gamma_r)} \frac{h_{\Omega^-}(\zeta)}{\zeta - z} d\zeta.
  \end{align*}
  \begin{corollary}  \label{co:outside_bound_J}
   Let $\Gamma$ be a quasicircle not containing $\infty$, bounding
   domains $\Omega^\pm$.  Then ${\Pi}_+(\Omega^\pm) = {\Pi}_-(\Omega^\pm)$.
   Furthermore, ${\Pi}_+(\Omega^\pm)$ and ${\Pi}_-(\Omega^\pm)$ are
   each bounded with respect to both norms $\mathcal{H}_+(\Gamma)$
   and $\mathcal{H}_-(\Gamma)$ on the domain.
  \end{corollary}
  \begin{proof}
   This follows from Theorem \ref{th:J_bounded}, (\ref{eq:norm_equivalent_in_out}) and Theorem \ref{th:contour_integral_jump}.
  \end{proof}

  \begin{remark}  Of course, the integral operators (\ref{eq:J_defn_contour}) define bounded operators
   directly from $\mathcal{D}_{\mathrm{harm}}(\Omega^\pm)$ with respect to either the semi-norms or norms.
  \end{remark}

 \begin{remark} In defining the operator $J(\Gamma)$, we ``broke the symmetry'' between the inside and
  outside by choosing to
  take the integral over $\Omega^+$.  Theorem \ref{th:contour_integral_jump}
  can be used to show that for quasicircles, $J(\Gamma)$ can be defined using the integral over $\Omega^-$
  \[  J(\Gamma)h(z) := h_{\Omega^-}(z) \chi_{\Omega^{-}}(z)+\frac{1}{\pi }\iint_{\Omega^{-}}
   \frac{\overline{\partial} h_{\Omega^-}(\zeta)}{\zeta -z}\, dA(\zeta), \,\,\, z\in \mathbb{C}\setminus \Gamma  \]
  with no change to the outcome (we leave the proof to the interested reader).
  This is obvious for highly regular curves by Stokes' theorem, but for quasicircles it is a surprisingly subtle point. The proof of Theorem \ref{th:contour_integral_jump}
  suggests that it is closely linked to the density of polynomials in the Dirichlet space of a quasicircle.
 \end{remark}

  \begin{remark}
   It is also possible to prove a version of Corollary \ref{cor:boundedness of P(omega)} where the pointed norm
   is taken on the outside.  However, to do so one must alter the normalization of
   the Cauchy kernel
   so that $P(\Omega^+)h(p)=0$ for some point $p \in \Omega^+$, rather than $P(\Omega^-)h(\infty)=0$
   as we have here.
  \end{remark}
\end{subsection}
\begin{subsection}{The jump decomposition on quasicircles}
 In this section we prove that the jump decomposition on quasicircles holds.  To do this, we will
 first show that $\mathcal{H}(\Gamma)$ is naturally isomorphic to $\mathcal{D}(\Omega^+) \oplus \mathcal{D}_*(\Omega^-)$.
 The actual jump formula follows easily.

 We require a theorem.
 \begin{theorem} \label{th:jump_expected} Let $\Gamma$ be a quasicircle not containing $\infty$, bounding domains $\Omega^\pm$.
  For any $h \in \mathcal{H}(\Gamma)$ such that $h_{\Omega^-} \in \mathcal{D}_*(\Omega^-)$, we have that
  \[  J(\Gamma) h(z) = \left\{ \begin{array}{cc} 0 & z \in \Omega^+ \\ - h_{\Omega^-}(z) & z \in \Omega^-
   \end{array} \right..    \]
  On the other hand, for any $h \in \mathcal{H}(\Gamma)$ such that $h_{\Omega^+} \in \mathcal{D}(\Omega^+)$
  \[  J(\Gamma) h(z) = \left\{ \begin{array}{cc}h_{\Omega^+}(z) & z \in \Omega^+ \\
      0 & z \in \Omega^- \end{array} \right..  \]
 \end{theorem}
 \begin{proof}
  We apply the classical Plemelj-Sokhotski jump formula on regular curves and use the limiting
  integral formula.  For a holomorphic function $h$ on a domain $\Omega$ bounded by a smooth Jordan
  curve $\gamma$, if $h$ extends continuously to $\gamma$ then the classical formula yields
  \[  \frac{1}{2 \pi i } \int_{\gamma} \frac{h(\zeta)}{\zeta - z} d\zeta = \left\{ \begin{array}{cc}
     h(z) & z \in \Omega \\ 0 & z \in \mathbb{C} \backslash \text{cl} \Omega  \end{array} \right.  \]
     and similarly for $h$ holomorphic in the complementary region and vanishing at $\infty$.

  Now assume that $h$ is the boundary values of some $h_{\Omega^+}$ in the sense of
  Osborn.  By Theorem \ref{th:contour_integral_jump}, fixing $z$ we have
  \begin{equation*}
   J(\Gamma) h (z)   = \lim_{r \nearrow 1} \frac{1}{2 \pi i} \int_{f(\gamma_r)} \frac{h_{\Omega^+}(\zeta)}{\zeta - z} d \zeta
    =  \left\{ \begin{array}{cc}h_{\Omega^+}(z) & z \in \Omega^+ \\
      0 & z \in \Omega^- \end{array} \right.
  \end{equation*}
  since for any $z \notin \Gamma$, the curve $f(\gamma_r)$ eventually lies either always inside or always outside
  $f(\gamma_r)$ on some interval $[R,1)$ for $R <1$, and thus the integral is independent of $r$; for any fixed $r \in [R,1)$
  we can apply the classical formula.  The other claim is proven similarly using the integral formula involving the
  conformal map $g:\disk^- \rightarrow \Omega^-$.
 \end{proof}

 We will need some identities to prove the jump decomposition.  Let $\Gamma$ be a
 quasicircle not containing $\infty$, bounding domains $\Omega^\pm$.
 We define the mapping
 \begin{align*}
  \mathbf{I}_f: \mathcal{D}_*(\disk^-) & \rightarrow \mathcal{D}_*(\Omega^-) \\
  h & \mapsto P(\Omega^-) C_{f^{-1}} h.
 \end{align*}
 There is a restriction to the boundary which we have suppressed for simplicity of notation.
 We know that this is a bounded map into $\mathcal{D}_*(\Omega^-)$ by Theorems \ref{th:J_bounded}
 and \ref{th:normal_comp_isometry}.
 \begin{theorem}  \label{th:If_without_reflection}
  Let $\Gamma$ be a Jordan curve not containing $\infty$,
  with bounded and unbounded components $\Omega^\pm$ respectively. Let $f:\disk^+ \rightarrow \Omega^+$
   be a conformal map. For any
  polynomial $h \in \mathbb{C}[1/z]$  with zero constant term,
  \[ \mathbf{I}_f h (z) = - \lim_{r \nearrow 1}  \frac{1}{2 \pi i}
   \int_{f(\gamma_r)} \frac{  h \circ f^{-1}(\zeta)}
   {\zeta - z} \, d\zeta.  \]
 \end{theorem}
  \begin{proof}  By linearity it is enough to verify this for monomials.
 Setting $h(z)=z^{-n}$ for $n>0$, it is easily verified that
 \[  (h \circ f^{-1})_{\Omega^+}(z)= \overline{h  \circ (1/f^{-1}(z))}    \]
 since the boundary values agree on $\mathbb{S}^1$ after composition by $f$.  Thus
 \begin{align*}
  (\mathbf{I}_f h)(z) & = -  \lim_{r \nearrow 1} \frac{1}{2 \pi i} \int_{f(\gamma_r)}
   \frac{ \overline{f^{-1}(\zeta)}^{n}}{\zeta - z}\, d\zeta \\
   & = -  \lim_{ r \nearrow 1} \frac{1}{2 \pi i} \int_{\gamma_r} \frac{\overline{\xi}^{n}}
   { f(\xi) - z} f'(\xi)\, d\xi \\
   & = -  \lim_{ r \nearrow 1} \frac{1}{2 \pi i} \int_{\gamma_r} \frac{{r^{2n}} \xi^{-n}}
   { f(\xi) - z} f'(\xi) \,d\xi \\
   & = -  \lim_{r \nearrow 1} \frac{1}{2 \pi i} \int_{\gamma_r} \frac{\xi^{-n}}
   {f(\xi) - z} f'(\xi) \,d\xi \\
   & = -  \lim_{r \nearrow 1} \frac{1}{2 \pi i} \int_{f(\gamma_r)} \frac{ h \circ f^{-1}(\zeta) }{\zeta - z}
    \,d\zeta
 \end{align*}
 where we have used the fact that the integral
 $\int_{\gamma_r} {\xi^{-n} \cdot  f'(\xi) } /
   {(f(\xi) - z)}d\xi$ is independent of $r$.
 \end{proof}

 The map $\mathbf{I}_f$ is closely related to Faber polynomials, whose definition we now recall.
 \begin{definition}  \label{de:Faber_polynomials}
  Let $f:\disk \rightarrow \mathbb{C}$
  be a holomorphic map which is one-to-one on a neighbourhood of $0$.  Let $p=f(0)$.  The $n$th Faber polynomial
  of $f$ is
  \begin{equation} \label{eq:Faber_defn_cs}
   \Phi_n(f)(z) = \frac{c_{-n}}{(z-p)^{n}} + \cdots + \frac{c_{-1}}{(z-p)}
  \end{equation}
  where the $c_n$ are defined by
  \[ (f^{-1}(z))^{-n}(z) = \sum_{k=-n}^\infty c_{k} (z-p)^k.  \]
 \end{definition}
 \begin{remark}  Usually, it is assumed that $p=0$ in the literature.
 Allowing $p \neq 0$ costs nothing and we will require it for use in future papers.
 It is also common to include the constant term in the expansion (\ref{eq:Faber_defn_cs})
 in the Faber polynomials.
 \end{remark}
 By Theorem \ref{th:If_without_reflection}, we have that
 \begin{equation}  \label{eq:If_and_Faber}
  \mathbf{I}_f(z^{-n}) = \Phi_n.
 \end{equation}

 The Faber polynomials satisfy the following identity.
 \begin{equation} \label{eq:Faber_identity}
  \Phi_n(f) \circ f(z) = z^{-n} + \sum_{k=0}^\infty \beta_k^n z^k
 \end{equation}
 for some coefficients $\beta_k^n$ referred to as the Grunsky coefficients \cite{Durenbook,Pommerenkebook}
 (usually in the case that $p=0$).
 To prove (\ref{eq:Faber_identity}), observe that $\Phi_n(f)(\zeta) = f^{-1}(\zeta)^{-n} - H(\zeta)$
 where $H = P(\Omega^+)({(f^{-1})}^{-n})$ is holomorphic on $\Omega^+$.
 Thus $\Phi_n(f)(f(z)) = z^{-n} - H \circ f(z)$ where $H \circ f$ is holomorphic on
 $\disk^+$.
 Since $f$ is holomorphic on $|z|<1$ and $\Phi_n(f)$ is holomorphic on $0<|z|<\infty$,
 the radius of convergence of the power series on the right hand side is greater
 than or equal to one.  This leads to the following Lemma.
 \begin{lemma}  \label{le:If_injective} Let $\Gamma$ be a quasicircle not containing
  $\infty$.   The map $P(\disk^-) C_f$ is a bounded left inverse of $\mathbf{I}_f$.  In particular,
  $\mathbf{I}_f$ is injective.
 \end{lemma}
 \begin{proof}
  By Theorem \ref{th:If_without_reflection} and equations (\ref{eq:If_and_Faber}) and (\ref{eq:Faber_identity}) this
  holds for polynomials in $\mathbb{C}[1/z]$  with zero constant term.  Since these are dense in $\mathcal{D}(\disk^-)$
  and $P(\disk^-) C_f$ and $\mathbf{I}_f$ are bounded, this completes the proof.
 \end{proof}

 We may finally prove the main result of the paper.
 \begin{theorem} \label{th:jump_decomposition_K}
  For a quasicircle $\Gamma$ not containing $\infty$, bounding domains $\Omega^\pm$, and $p \in \Omega^+$
 define
 \begin{align*}
  K:\mathcal{H}(\Gamma) & \rightarrow \mathcal{D}(\Omega^+) \oplus \mathcal{D}_*(\Omega^-) \\
  h & \mapsto (P(\Omega^+) h , P(\Omega^-) h).
 \end{align*}
  Then $K$ is a bounded
  isomorphism with respect to the norms $\| \cdot \|_{\mathcal{H}_p(\Gamma)}$ and $\| \cdot \|_{\mathcal{D}_p(\Omega^+) \oplus
  \mathcal{D}(\Omega^-)}$.
 \end{theorem}
 \begin{proof}
  The fact that $K$ is bounded follows immediately from either Theorem \ref{th:J_bounded} or Corollary \ref{cor:boundedness of P(omega)},
  depending on the desired choice of norms.   To see that it is
  surjective, given $(h_+,h_-) \in \mathcal{D}_p(\Omega^+) \oplus \mathcal{D}_*(\Omega^-)$, let $\tilde{h}_+$
  and $\tilde{h}_-$ denote their boundary values in $\mathcal{H}(\Gamma)$ and set $h = \tilde{h}_+ + \tilde{h}_-$.  By Theorem
  \ref{th:jump_expected}, we have that $P(\Omega^\pm)h = h_\pm$ which proves the claim.

  To show that $K$ is injective, choose a conformal map $f:\disk^+ \rightarrow \Omega^+$
   such that $f(0)=p$. Observe that $C_{f^{-1}}$ is an isometry from $\mathcal{D}_0(\disk^+)$ to $\mathcal{D}_p(\Omega^+)$
   by Theorem \ref{th:normal_comp_isometry}.  Hence it is an isometry from $\mathcal{H}_0(\mathbb{S}^1)$ to
   $\mathcal{H}_p(\Gamma)$.  Thus any $h \in \mathcal{H}_p(\Gamma)$ can be written $C_{f^{-1}} ( H_+ + H_-)$
   for $H_+ \in \mathcal{D}_0(\disk^+)$ and $H_- \in \mathcal{D}_*(\disk^-)$.    Now
   \begin{equation} \label{eq:Kh_formula}
    Kh = \left( P(\Omega^+) C_{f^{-1}} H_+ + P(\Omega^+) C_{f^{-1}} H_-, P(\Omega^-) C_{f^{-1}} H_- \right).
   \end{equation}
   If $Kh=0$ then $\mathbf{I}_f H_- = P(\Omega^-) C_{f^{-1}} H_- =0$ so by Lemma \ref{le:If_injective}
   $H_-=0$.  But then
   \[  0=  P(\Omega^+) C_{f^{-1}} H_+ + P(\Omega^+) C_{f^{-1}} H_- = P(\Omega^+) C_{f^{-1}} H_+. \]
   Since $C_{f^{-1}} H_+ \in \mathcal{D}_p(\Omega^+)$ we have that $C_{f^{-1}} H_+ = P(\Omega^+) C_{f^{-1}} H_+ =0$
   but $C_{f^{-1}}$ is an isometry so $H_+=0$.  Thus $h=0$ so $K$ is injective as claimed.
 \end{proof}
 \begin{remark}  It is not hard to show that, conversely, $K$ injective implies that $\mathbf{I}_f$ is injective
  directly from the definition of $K$.
 \end{remark}
  \begin{remark}  $K$ is also bounded with respect to the semi-norms on
  $\mathcal{H}(\Gamma)$ and $\mathcal{D}(\Omega^\pm)$ by Theorem \ref{th:J_bounded}.
 \end{remark}

Finally, as a consequence of our previous results we establish the well-posedness of the following Rieman-Hilbert problem which was posed in less precise terms in the introduction of this paper.

\begin{theorem} \label{th:Cauchy_bounded}
Let $\Gamma$ be a quasicircle not containing $\infty$ bounding $\Omega^\pm$ and let $h\in \mathcal{H}(\Gamma)$. Then there exist unique functions $u_+ \in \mathcal{D}(\Omega^{+})$, $u_- \in \mathcal{D}_*(\Omega^{-})$  such that the boundary values $h_{\pm}$ of $u_{\pm}$  are defined in the sense of Osborn and $h_{+}(w)+h_{-}(w)= h(w)$ except on a set of harmonic
 measure zero with respect to both $\Omega^+$ and $\Omega^-$. Moreover the solutions $u_\pm$ depend continuously on $h$.
\end{theorem}

\begin{remark}  The continuity of $u_\pm$ is with respect to either the pointed norms $\mathcal{H}_p(\Gamma)$, $\mathcal{D}_p(\Omega^+)$,
and $\mathcal{D}(\Omega^-)$,  or to the unpointed norms $\mathcal{H}(\Gamma)$, $\mathcal{D}(\Omega^+)$, and $\mathcal{D}(\Omega^-)$.
\end{remark}

 \begin{proof}  We claim that $u_\pm = P(\Omega^\pm)h$ is the unique solution to the problem. Indeed, $u_+ \in \mathcal{D}(\Omega^{+})$
 and $u_- \in \mathcal{D}_*(\Omega^-)$ by Theorem \ref{th:J_bounded}, which also shows the continuous dependence of the solutions on the initial data $h$ with respect to the semi-norms.  The continuous dependence for the pointed norms
 follows from Corollary \ref{cor:boundedness of P(omega)}.  By Theorem \ref{th:jump_expected}, $K h_+ = (u_+,0)$, so
  $K (h-h_+) = (0,u_-)$.  Since $K h_- = (0,u_-)$ and $K$ is injective we must have
  that $h-h_+ =h_-$ which proves the Sokhotski-Plemelj jump relation. The uniqueness of the solution is a consequence of Theorem \ref{th:jump_decomposition_K}.
 \end{proof}

\end{subsection}
\end{section}

\end{document}